\documentclass[a4paper,11pt]{amsart}

\usepackage[margin=3cm]{geometry}
\usepackage{amsmath,amssymb,amsthm,mathtools,mathrsfs,esint}
\usepackage{enumitem}
\usepackage{xcolor}
\usepackage{microtype}
\usepackage[colorlinks=true,linkcolor=blue!55!black,citecolor=blue!55!black,urlcolor=blue!55!black]{hyperref}

\newif\ifmarked
\ifdefined\BuildMarked
  \markedtrue
\else
  \markedfalse
\fi

\ifmarked
  \newcommand{\authornote}[1]{%
    \par\medskip\noindent
    \fcolorbox{red!65!black}{yellow!10}{%
      \parbox{0.94\linewidth}{\small\color{red!65!black}\textbf{Author check.} #1}}%
    \par\medskip}
  \newcommand{\newresult}[1]{\textcolor{blue!65!black}{#1}}
\else
  \newcommand{\authornote}[1]{}
  \newcommand{\newresult}[1]{#1}
\fi

\newtheorem{theorem}{Theorem}[section]

\newtheorem{lemma}[theorem]{Lemma}
\newtheorem{corollary}[theorem]{Corollary}

\theoremstyle{remark}
\newtheorem{remark}[theorem]{Remark}

\newcommand{\Ric}{\operatorname{Ric}}
\newcommand{\tr}{\operatorname{tr}}
\newcommand{\II}{\mathrm{II}}

\newcommand{\Id}{\operatorname{Id}}
\newcommand{\Sph}{\mathbb S}

\newcommand{\gbar}{\overline g}
\newcommand{\Rlin}{\mathscr R}

\title[Conformal boundary deformations]{Conformal Boundary Deformations under Ricci Lower Bounds: Eigenvalue Counterexamples}

\author[F. Li]{Fagui Li}
\address{Frontier Interdisciplinary Domain, Beijing Institute of Technology, Zhuhai, Guangdong 519088, P. R. China}
\email{lifagui@bitzh.edu.cn}

\author[Y. Zhao]{Yuhang Zhao${}^{*}$}
\thanks{${}^{*}$Corresponding author.}
\address{School of Mathematics, Nanjing University, Nanjing 210093, P. R. China}
\email{yuhangzhao@smail.nju.edu.cn}
\date{August 25, 2026}

\begin{document}

\begin{abstract}
Let $(M^{n+1},g)$ be a compact Riemannian manifold with boundary. Under the assumptions $\Ric_g\geq ng$ and $\II_g\geq0$,  Wang proposed  a sharp strengthening of the Choi--Wang--Reilly estimate,  asserting that the first nonzero Laplace eigenvalue of the boundary is at least $n$; see [J. Geom. Anal. 31 (2021)].  We disprove this assertion in every dimension $n+1\geq3$.  More precisely, we construct a sequence of metrics on the hemisphere $\Sph^{n+1}_{+}$ converging in $C^\infty$ to the round metric and satisfying
\[
 \Ric_g>n g,\qquad \II_g>0,\qquad
 \lambda_1(\partial\Sph^{n+1}_{+},g|_{\partial\Sph^{n+1}_{+}})<n.
\]
The construction starts from Zhu's infinitesimal conformal deformation, which lowers one branch of the first boundary eigenspace while preserving the normalized Ricci lower bound to first order.  We add a multiple of the spherical height function. 

\end{abstract}

\maketitle

\authornote{The clean version contains no annotations.  This marked version records the principal proof checks and the literature-priority issue that should be revisited immediately before submission.}

\section{Introduction}

Let $(M^{n+1},g)$ be a compact connected Riemannian manifold with smooth nonempty boundary $\Sigma$.  We use the convention
\[
 \II(X,Y)=g(\nabla_X\nu,Y),\qquad X,Y\in T\Sigma,
\]
where $\nu$ is the outward unit normal.  Thus the boundary of a Euclidean ball is strictly convex.

The interaction between lower Ricci curvature, boundary convexity, and intrinsic boundary geometry is naturally studied through Reilly's formula \cite{Reilly1977}.  If a first eigenfunction on the boundary is extended harmonically to the interior, the formula compares its boundary Dirichlet energy with the interior Hessian and Ricci terms.  This is the one-sided counterpart of the argument of Choi and Wang for embedded minimal hypersurfaces \cite{ChoiWang1983}.  In the notation used here, Xiaodong Wang recorded the consequence
\begin{equation}\label{eq:Wang-half}
 \Ric_g\geq n g,\qquad \II_g\geq0
 \quad\Longrightarrow\quad
 \lambda_1(\Sigma)\geq \frac n2,
\end{equation}
and asked whether the optimal lower bound is $n$ \cite[Proposition~7 and Remark~2]{Wang2021}.  The round hemisphere realizes the proposed value.  The factor-two improvement also mirrors Yau's first-eigenvalue conjecture for closed embedded minimal hypersurfaces of the round sphere \cite[Problem~100]{Yau1982}: the coordinate functions give the eigenvalue $n$, while the Choi--Wang argument gives only $n/2$.  Recent quantitative refinements for embedded minimal hypersurfaces and sharp boundary-Laplacian estimates under additional geometric assumptions appear in \cite{DuncanSireSpruck2024,LiuYangBoundary2026}.

Related rigidity and spectral questions for manifolds with nonnegative Ricci curvature and convex boundary provide a broader context.  For boundary rigidity, convexity, and related comparison results, see \cite{HHWang1999,Xia1997,MiaoWang2016,GuoHangWang2021,YanZhu2026}; in low-dimensional sharp comparisons, Hersch's inequality is also a standard ingredient \cite{Hersch1970}.  A parallel line of work concerns the first nonzero Steklov eigenvalue, beginning with Escobar's estimates and conjecture \cite{Escobar1997,Escobar1999}, with subsequent results under stronger geometric assumptions in \cite{XiaXiong2024,DuncanKumar2025,LiuYangSteklov2026}.  Recent conformal counterexamples to the original sharp Steklov conjecture under positive Ricci curvature and strict boundary convexity were obtained by Sun, Wang, and Wang \cite{SunWangWang2026}.

Zhu \cite{Zhu2017} showed that the analogous statement fails if convexity is replaced by minimality: for every $n\geq2$, he constructed a metric on $\Sph^{n+1}_{+}$ satisfying $\Ric\geq ng$, with minimal boundary and with first boundary eigenvalue smaller than $n$.  His proof first produces an explicit infinitesimal conformal factor and then corrects the boundary mean curvature by a diffeomorphism and a higher-order conformal term.  The final second fundamental form is not asserted to be nonnegative.  Thus Zhu's theorem does not by itself decide whether the convexity hypothesis in \eqref{eq:Wang-half} forces the proposed sharp bound.  This distinction is essential: if the induced boundary metric is exactly the standard sphere, the Ricci lower bound and convexity do force hemispherical rigidity by Hang and Wang \cite{HangWang2009}.

The main observation of this note is that the infinitesimal conformal factor in Zhu's construction has a boundary-invisible freedom.  Let $z$ be the spherical height function on the hemisphere, normalized by
\[
 z=0\quad\text{on }\Sigma,\qquad \partial_\nu z=-1\quad\text{on }\Sigma.
\]
The conformal metric variation $-z\gbar$ is the restriction of an ambient pure-diffeomorphism direction for the round metric and hence belongs to the kernel of the linearized normalized Ricci tensor; its generating vector field is not tangent to the boundary.  At the same time, $z$ vanishes on the equator and $-z$ has positive outward normal derivative.  Adding a sufficiently large multiple of this mode therefore has three simultaneous effects:
\begin{enumerate}[label=(\roman*),leftmargin=2.2em]
 \item it does not change the induced boundary metric;
 \item it does not change the first-order Ricci inequality;
 \item it makes the boundary strictly convex under the resulting conformal deformation.
\end{enumerate}
This yields the following theorem.

\begin{theorem}\label{thm:main}
\newresult{For every integer $n\geq2$, there is a sequence of smooth metrics $g_j$ on the hemisphere $\Sph^{n+1}_{+}$ converging in $C^\infty$ to the round metric such that}
\begin{equation*}
 \newresult{\Ric_{g_j}>n g_j,\qquad
 \II_{g_j}>0,\qquad
 \lambda_1\bigl(\partial\Sph^{n+1}_{+},g_j|_{\partial\Sph^{n+1}_{+}}\bigr)<n.}
\end{equation*}
Moreover, the boundary is umbilic for each $g_j$.
\end{theorem}

In the ambient-dimensional notation of \cite{Wang2021}, Theorem~\ref{thm:main} disproves the proposed lower bound $\lambda_1(\partial M)\geq \dim M-1$ for every $\dim M\geq3$, even when the Ricci inequality and boundary convexity are both strict.  The lower bound $n/2$ in \eqref{eq:Wang-half} is not contradicted; our perturbations show that the missing factor $2$ cannot be recovered from these one-sided hypotheses alone.

\begin{remark}
    When $n=1$, the hypotheses become $K\geq1$ and nonnegative boundary geodesic curvature. By \cite[Theorem~4]{HangWang2009},  the first nonzero eigenvalue satisfies $\lambda_1(\partial M)\geq1$. Thus,  counterexamples occur only when  $n \geq 2$.
\end{remark}
\begin{remark}\label{rem:hang-wang}
There is no conflict with the boundary rigidity theorem of Hang and Wang \cite{HangWang2009}.  That theorem additionally assumes that the induced boundary metric is the standard round metric.  Here the induced metric is deformed in precisely the spectral direction inherited from Zhu's conformal factor.
\end{remark}

The paper is organized as follows.  Section~\ref{sec:prelim} records the conformal variation formulas, the perturbation-theoretic point needed for the multiple first eigenvalue, and the precise input from Zhu.  Section~\ref{sec:mode} identifies the Ricci-neutral convexifying mode, and Theorem~\ref{thm:main} is proved in Section~\ref{sec:proof-main}.

\section{Conformal variations on the round hemisphere}\label{sec:prelim}

Let $(\Sph^{n+1}_{+},\gbar)$ be the standard unit hemisphere, let
\[
 \Sigma=\partial\Sph^{n+1}_{+}\cong\Sph^n,
\]
and let $\nu$ denote the outward unit normal along $\Sigma$.  The round metric satisfies
\[
 \Ric_{\gbar}=n\gbar,
 \qquad
 \II_{\gbar}=0,
 \qquad
 \lambda_1(\Sigma,\gbar|_\Sigma)=n.
\]
We use the Laplacian sign convention $\Delta=\tr\nabla^2$, so that $\Delta x_i=-n x_i$ on the unit $n$-sphere.
All inequalities between symmetric two-tensors are understood pointwise as inequalities of quadratic forms.

\subsection{Ricci and boundary variations}

Consider a conformal family
\begin{equation}\label{eq:conformal-family}
 g_t=e^{t\varphi}\gbar,
 \qquad \varphi\in C^\infty(\Sph^{n+1}_{+}).
\end{equation}
Its initial metric variation is $h=\varphi\gbar$.  Since $\gbar$ is Einstein, the variation of the normalized Ricci endomorphism is equivalent to the variation of $\Ric_{g_t}-ng_t$.  The standard conformal formula gives
\begin{equation}\label{eq:lin-ricci}
 -2\left.\frac{d}{dt}\right|_{t=0}
   \bigl(\Ric_{g_t}-ng_t\bigr)
 =\bigl(\Delta\varphi+2n\varphi\bigr)\gbar
  +(n-1)\nabla^2\varphi.
\end{equation}
Following Zhu, define the linearized normalized Ricci tensor by
\begin{equation*}
 \Rlin(\varphi)
 :=\left.\frac{d}{dt}\right|_{t=0}
     \bigl(\Ric_{g_t}-ng_t\bigr).
\end{equation*}
More explicitly, because $\Ric_{\gbar}=n\gbar$,
\begin{equation}\label{eq:endomorphism-covariant-translation}
 \left.\frac{d}{dt}\right|_{t=0}
 \bigl(g_t^{-1}\Ric_{g_t}-n\Id\bigr)
 =\gbar^{-1}\Rlin(\varphi).
\end{equation}
Thus Zhu's endomorphism inequality and the covariant inequality $\Rlin(\varphi)\geq0$ are identical at the round metric.  For later reference, the full conformal formula along \eqref{eq:conformal-family} is
\begin{align*}
 \Ric_{g_t}
 &=n\gbar-\frac{n-1}{2}t\nabla^2\varphi
   -\frac12t(\Delta\varphi)\gbar\\
 &\quad+\frac{n-1}{4}t^2
  \bigl(d\varphi\otimes d\varphi-|\nabla\varphi|^2\gbar\bigr),
\end{align*}
where all derivatives on the right are taken with respect to $\gbar$.

For a general conformal change $\widetilde g=e^w g$, the second fundamental form transforms as
\begin{equation}\label{eq:II-conformal}
 \II_{\widetilde g}
 =e^{w/2}\left(\II_g+\frac12(\partial_\nu w)g|_{T\Sigma}\right),
\end{equation}
where the derivative and the normal on the right are taken with respect to $g$.  Since the equator is totally geodesic, \eqref{eq:II-conformal} yields, for \eqref{eq:conformal-family},
\begin{equation}\label{eq:II-family}
 \II_{g_t}
 =\frac t2 e^{t\varphi/2}(\partial_\nu\varphi)\gbar|_{T\Sigma}.
\end{equation}
In particular, if $\partial_\nu\varphi>0$ on $\Sigma$, then the boundary is strictly convex and umbilic for every sufficiently small $t>0$.

\subsection{Analytic splitting of the first boundary eigenspace}

The first nonzero eigenvalue $n$ of the round $\Sph^n$ has multiplicity $n+1$.  For a real-analytic family of boundary metrics, Rellich--Kato perturbation theory gives $n+1$ real-analytic eigenvalue branches issuing from this cluster \cite{Berger1973,Kato1995}.  The neighboring eigenvalue clusters are separated from $n$, so each branch remains a positive eigenvalue for small $|t|$.  If one branch has negative derivative at zero, then the variational ordering of the spectrum gives
\begin{equation*}
 \lambda_1(\Sigma,\widehat g_t)
 \leq n-\gamma t+O(t^2)
\end{equation*}
for some $\gamma>0$ and all small $t>0$.  We will only use this one-sided consequence.

The following key lemma follows from Zhu's explicit conformal-factor construction \cite{Zhu2017}.

\begin{lemma}\label{lem:Zhu-input}
For every $n\geq2$, there exists $f\in C^\infty(\Sph^{n+1}_{+})$ such that:
\begin{enumerate}[label=(\roman*),leftmargin=2.2em]
 \item $\Rlin(f)\geq0$ on $\Sph^{n+1}_{+}$;
 \item for the boundary metric family
 \[
  \widehat g_t^{\,f}=e^{t f|_\Sigma}\gbar|_\Sigma,
 \]
 an analytic eigenvalue branch issuing from $n$ has strictly negative derivative.  Equivalently, there is a constant $\gamma>0$ such that
 \begin{equation}\label{eq:Zhu-branch}
  \lambda_1(\Sigma,\widehat g_t^{\,f})
  \leq n-\gamma t+O(t^2).
 \end{equation}
\end{enumerate}
\end{lemma}

\begin{proof}
Let $E_n=\ker(\Delta_{\Sph^n}+n)$, which has dimension $n+1$.  Proposition~4.1 of Zhu \cite{Zhu2017} supplies a smooth $f$ for which
\[
 \left.\frac{d}{dt}\right|_{t=0}
 \bigl(g_t^{-1}\Ric_{g_t}\bigr)\geq0
\]
whenever $g'_0=f\gbar$.  By \eqref{eq:endomorphism-covariant-translation}, this is exactly $\Rlin(f)\geq0$.

The same proposition computes the self-adjoint compression to $E_n$ of the first variation of the boundary Laplacian.  In Zhu's sign convention, its eigenvalues give the negatives of the derivatives of the positive Laplace eigenvalues, and the calculation exhibits one corresponding eigenvalue derivative $\mu<0$.  Apply this statement to the analytic conformal family $g_t=e^{tf}\gbar$.  Rellich--Kato theory, in the form used in \cite[Section~4.2]{Zhu2017}, produces an analytic branch
\[
 \lambda(t)=n+\mu t+O(t^2).
\]
Since $\mu<0$, setting $\gamma=-\mu>0$ gives \eqref{eq:Zhu-branch}.  This invariant formulation is independent of the choice and indexing of a basis for $E_n$; only the existence of a negative eigenvalue of the first-variation form is used.
\end{proof}

\authornote{Proof check.  This proof now explicitly translates Zhu's endomorphism inequality into the covariant tensor inequality used here and derives the decreasing analytic branch without copying the inconsistent basis indices in the displayed matrix of Zhu's Proposition~4.1.}

\section{A boundary-invisible, Ricci-neutral convexifying mode}\label{sec:mode}

Let $r$ be the distance from the north pole of the round hemisphere and define
\begin{equation*}
 z=\cos r.
\end{equation*}
The function $z$ is the restriction of a linear coordinate function on the round sphere.  It satisfies
\begin{equation}\label{eq:z-identities}
 \nabla^2z=-z\gbar,
 \qquad
 \Delta z=-(n+1)z,
 \qquad
 z|_\Sigma=0,
 \qquad
 \partial_\nu z|_\Sigma=-1.
\end{equation}

\begin{lemma}\label{lem:kernel-mode}
For every $f\in C^\infty(\Sph^{n+1}_{+})$ and every constant $C\in\mathbb R$, set
\begin{equation*}
 f_C=f-Cz.
\end{equation*}
Then:
\begin{enumerate}[label=(\roman*),leftmargin=2.2em]
 \item $f_C|_\Sigma=f|_\Sigma$;
 \item $\Rlin(f_C)=\Rlin(f)$;
 \item $\partial_\nu f_C=\partial_\nu f+C$ on $\Sigma$.
\end{enumerate}
Consequently, $C$ can be chosen so that $\partial_\nu f_C>0$ everywhere on $\Sigma$, without changing either the induced boundary metric family or the first-order normalized Ricci tensor.
\end{lemma}

\begin{proof}
The first and third assertions follow immediately from the boundary identities in \eqref{eq:z-identities}.  For the second assertion, apply \eqref{eq:lin-ricci} to $-Cz$.  Since
\[
 \Delta(-Cz)=-(n+1)(-Cz),
 \qquad
 \nabla^2(-Cz)=-(-Cz)\gbar,
\]
we have
\begin{align*}
 &\bigl(\Delta(-Cz)+2n(-Cz)\bigr)\gbar
 +(n-1)\nabla^2(-Cz)\\
 &\qquad=(n-1)(-Cz)\gbar-(n-1)(-Cz)\gbar=0.
\end{align*}
Thus $\Rlin(-Cz)=0$, and linearity gives $\Rlin(f_C)=\Rlin(f)$.

Finally, because $\Sigma$ is compact, choose
\[
 C> -\min_\Sigma\partial_\nu f.
\]
Then $\partial_\nu f_C>0$ on $\Sigma$.
\end{proof}

The geometric meaning of the kernel calculation is worth recording.

\begin{remark}\label{rem:pure-gauge}
Since $\nabla^2z=-z\gbar$,
\begin{equation*}
 -Cz\gbar=\frac C2\mathcal L_{\nabla z}\gbar.
\end{equation*}
Thus the added metric variation is a pure diffeomorphism direction for the round Einstein metric.  The generating vector field is normal rather than tangent to the equator.  Hence the variation is invisible to the induced boundary metric at first order---and, for the conformal family used below, exactly invisible---while it changes the extrinsic geometry of the fixed boundary hypersurface.  The fact that the generating flow does not preserve that hypersurface is precisely the mechanism that produces strict convexity.
\end{remark}

\begin{corollary}\label{cor:exact-boundary}
For $f_C$ as in Lemma~\ref{lem:kernel-mode}, the conformal families
\[
 e^{t f_C}\gbar
 \quad\text{and}\quad
 e^{t f}\gbar
\]
induce exactly the same metric on $\Sigma$ for every $t$ for which they are defined.
\end{corollary}

\begin{proof}
This follows from $f_C|_\Sigma=f|_\Sigma$.
\end{proof}

\section{Proof of   Theorem \ref{thm:main}}\label{sec:proof-main}

We first isolate the elementary passage from a first-order Ricci inequality to an exact inequality after a quadratic rescaling.

\begin{lemma}\label{lem:quadratic-recovery}
Let $t\mapsto g_t$ be a $C^2$ curve, in the $C^2$ topology, of smooth metrics on a compact manifold, possibly with boundary, with
\[
 g_0=\gbar,
 \qquad
 \Ric_{\gbar}=n\gbar,
 \qquad
 \left.\frac{d}{dt}\right|_{t=0}(\Ric_{g_t}-ng_t)\geq0.
\]
Then there are constants $K>0$ and $t_0>0$ such that
\begin{equation}\label{eq:quadratic-ricci}
 \Ric_{g_t}\geq(n-Kt^2)g_t
 \qquad\text{for }0\leq t\leq t_0.
\end{equation}
If $L>K/n$ and
\begin{equation*}
 \widetilde g_t=(1-Lt^2)g_t,
\end{equation*}
then, after decreasing $t_0$ if necessary,
\begin{equation*}
 \Ric_{\widetilde g_t}>n\widetilde g_t
 \qquad\text{for }0<t\leq t_0.
\end{equation*}
\end{lemma}

\begin{proof}
Set $S_t=\Ric_{g_t}-ng_t$.  Taylor's theorem in the $C^0$ norm of symmetric two-tensors gives
\[
 S_t=tS'_0+t^2R_t,
\]
where $R_t$ is uniformly bounded with respect to $\gbar$ for $t$ small.  Since $S'_0\geq0$, there is $K_0>0$ such that
\[
 S_t\geq-K_0t^2\gbar.
\]
The metrics $g_t$ and $\gbar$ are uniformly equivalent for small $t$, so, after enlarging the constant, this becomes
\[
 S_t\geq-Kt^2g_t,
\]
which is \eqref{eq:quadratic-ricci}.

Decrease $t_0$ so that $1-Lt^2>0$.  Constant rescaling does not change the Ricci tensor as a covariant two-tensor.  Consequently,
\begin{align*}
 \Ric_{\widetilde g_t}-n\widetilde g_t
 &=\Ric_{g_t}-ng_t+nLt^2g_t\\
 &\geq(nL-K)t^2g_t>0
\end{align*}
for $0<t\leq t_0$, because $nL>K$.
\end{proof}

\begin{proof}[\textbf{Proof of Theorem~\ref{thm:main}}]
Fix $n\geq2$ and choose the function $f$ supplied by Lemma~\ref{lem:Zhu-input}.  By Lemma~\ref{lem:kernel-mode}, choose $C$ so large that
\begin{equation}\label{eq:normal-positive}
 f_C=f-Cz,
 \qquad
 \partial_\nu f_C>0\quad\text{on }\Sigma.
\end{equation}
Consider the analytic conformal family
\begin{equation*}
 g_t=e^{t f_C}\gbar.
\end{equation*}

Because $\Rlin(f_C)=\Rlin(f)\geq0$, the family $g_t$ satisfies the hypothesis of Lemma~\ref{lem:quadratic-recovery}.  Hence, for suitable $K>0$,
\begin{equation*}
 \Ric_{g_t}\geq(n-Kt^2)g_t.
\end{equation*}
Choose $L>K/n$ and define
\begin{equation}\label{eq:gtilde-main}
 \widetilde g_t=(1-Lt^2)g_t.
\end{equation}
For all sufficiently small $t>0$, Lemma~\ref{lem:quadratic-recovery} gives
\begin{equation*}
 \Ric_{\widetilde g_t}>n\widetilde g_t.
\end{equation*}

Next, the conformal boundary formula \eqref{eq:II-family} and \eqref{eq:normal-positive} imply
\begin{equation*}
 \II_{g_t}
 =\frac t2e^{tf_C/2}(\partial_\nu f_C)\gbar|_{T\Sigma}>0.
\end{equation*}
Thus the boundary is strictly convex and umbilic.  Under the constant scaling in \eqref{eq:gtilde-main},
\[
 \II_{\widetilde g_t}=\sqrt{1-Lt^2}\,\II_{g_t},
\]
so strict convexity and umbilicity persist.

It remains to check the boundary eigenvalue.  By Corollary~\ref{cor:exact-boundary},
\begin{equation*}
 g_t|_{T\Sigma}=e^{t f|_\Sigma}\gbar|_{T\Sigma}.
\end{equation*}
Consequently, Lemma~\ref{lem:Zhu-input} gives an eigenvalue branch satisfying
\begin{equation*}
 \lambda_1(\Sigma,g_t|_{T\Sigma})
 \leq n-\gamma t+O(t^2)
\end{equation*}
for some $\gamma>0$.  Boundary Laplace eigenvalues scale inversely under constant scaling, and hence
\begin{align*}
 \lambda_1(\Sigma,\widetilde g_t|_{T\Sigma})
 &=(1-Lt^2)^{-1}\lambda_1(\Sigma,g_t|_{T\Sigma})\\
 &\leq n-\gamma t+O(t^2)<n,
\end{align*}
for all sufficiently small $t>0$; here we used $(1-Lt^2)^{-1}=1+Lt^2+O(t^4)$.

Finally, $\widetilde g_t\to\gbar$ in $C^\infty$ as $t\downarrow0$.  Taking any sequence $t_j\downarrow0$ for which the preceding inequalities hold and setting $g_j=\widetilde g_{t_j}$ proves the theorem.
\end{proof}

\authornote{Core proof check: the rescaling is of order $t^2$, whereas the boundary eigenvalue decreases by a strictly negative term of order $t$.  This order separation is essential and should be retained explicitly.}

\begin{corollary}\label{cor:Wang-false}
For every integer $m\geq3$, there exists a smooth compact $m$-dimensional Riemannian manifold $(M^m,g)$ with strictly convex boundary such that
\[
 \Ric_g>(m-1)g,
 \qquad
 \lambda_1(\partial M,g|_{\partial M})<m-1.
\]
In particular, the proposed sharp improvement of \cite[Proposition~7]{Wang2021} is false in every dimension $m\geq3$.
\end{corollary}

\begin{proof}
Apply Theorem~\ref{thm:main} with $n=m-1$ and $M=\Sph^m_+$.
\end{proof}

\begin{remark}[Why dimension two is different]\label{rem:dim-two}
When $\dim M=2$, the hypotheses become $K\geq1$ and nonnegative boundary geodesic curvature.  The boundary comparison theorem forces the boundary to be connected and gives $L(\partial M)\leq2\pi$ \cite[Theorem~4]{HangWang2009}; see also \cite{Toponogov1959}.  The first nonzero eigenvalue of a circle of length $L$ is $(2\pi/L)^2$, so $\lambda_1(\partial M)\geq1$.  The counterexamples above begin precisely in ambient dimension three, corresponding to Zhu's range $n\geq2$.
\end{remark}

\begin{remark}\label{rem:further}
The proof of Theorem~\ref{thm:main} separates three geometric roles that are often entangled in boundary deformation problems.  Zhu's conformal factor \cite{Zhu2017} controls the boundary spectrum and the first-order Ricci lower bound.  The spherical height mode controls the second fundamental form without changing either of those data to the required order.  Finally, a constant rescaling of quadratic size converts the infinitesimal Ricci inequality into the exact normalized inequality.  The mismatch between the linear spectral gain and quadratic rescaling loss is what makes the construction stable.

It is natural to define the optimal universal constant
\[
 c_n:=
 \inf_{\substack{(M^{n+1},g)\ \text{compact and connected},\ \partial M\ne\varnothing\\
                   \Ric_g\geq ng,\ \II_g\geq0}}
 \lambda_1(\partial M,g|_{\partial M}).
\]
The Reilly argument \cite{Reilly1977} gives $c_n\geq n/2$, while Theorem~\ref{thm:main} gives $c_n<n$.  Determining $c_n$ remains open.  The construction here is perturbative and does not determine how close to $n/2$ the first boundary eigenvalue can be driven under strict convexity.
\end{remark}

   {\noindent \bf  ORCID} Fagui Li:
   https://orcid.org/0000-0002-6733-7164
\\

 {\noindent \bf  ORCID} Yuhang Zhao:
   https://orcid.org/0009-0002-8812-0594

\end{document}